\documentclass[11pt]{amsart}
\theoremstyle{plain}
\newtheorem{thm}{Theorem}[section]
\newtheorem{theorem}[thm]{Theorem}

\newtheorem{lemma}[thm]{Lemma}
\newtheorem{corollary}[thm]{Corollary}
\newtheorem{proposition}[thm]{Proposition}
\theoremstyle{definition}
\newtheorem{remark}[thm]{Remark}

\newtheorem{definition}[thm]{Definition}

\newtheorem{example}[thm]{Example}

\numberwithin{equation}{section}

\title[Least negative intersection of positive closed currents]{Least negative intersections of positive closed currents on compact K\"ahler manifolds}

\author{Tuyen Trung Truong}

\address{Department of Mathematics, Syracuse University, Syracuse NY 13244} \email{tutruong@syr.edu}

\thanks{}

\begin{document}

\maketitle

\begin{abstract}
Let $X$ be a compact K\"ahler manifold of dimension $k$. Let $R$ be a positive closed $(p,p)$ current on $X$, and $T_1,\ldots ,T_{k-p}$ be positive closed $(1,1)$ currents on $X$. We define a so-called least negative intersection of the currents $T_1,T_2,\ldots ,T_{k-p}$ and $R$, as a sublinear bounded operator
\begin{eqnarray*}
\bigwedge (T_1,\ldots ,T_{k-p},R):~C^0(X)\rightarrow \mathbb{R}.
\end{eqnarray*} 
This operator is {\bf symmetric} in $T_1,\ldots ,T_{k-p}$. It is {\bf independent} of the choice of a quasi-potential $u_i$ of $T_i$, of the choice of a smooth closed $(1,1)$ form $\theta _i$ in the cohomology class of $T_i$, and of the choice of a K\"ahler form on $X$. Its total mass $<\bigwedge (T_1,\ldots ,T_{k-p},R),1>$ is the intersection in cohomology $\{T_1\}\{T_2\}\ldots \{T_{k-p}\}.\{R\}$. It has a semi-continuous property concerning approximating $T_i$ by appropriate smooth closed $(1,1)$ forms, plus some other good properties.

If $p=0$ and $T_1=\ldots =T_k=T$, we have a least negative Monge-Ampere operator $MA(T)=\bigwedge (T,\ldots ,T)$. If the set where $T$ has positive Lelong numbers does not contain any curve, then $MA(T)$ is positive. Several examples are given. 
\end{abstract}

\section{Introduction}

Wedge products of positive closed currents is a topic of great interest. One of its applications is to define the  Monge-Ampere operator for positive closed $(1,1)$ currents. The latter has many applications in complex geometry, among them is the celebrated solution of Yau \cite{yau} on the Calabi's conjecture. 

The seminal work of Bedford-Taylor \cite{bedford-taylor} defined wedge products of the form $dd^cu_1\wedge \ldots \wedge dd^cu_j \wedge R$, where $u_1,\ldots ,u_j$ are locally bounded psh functions and $R$ is a positive closed current. There have been many developments with important contributions from Kolodziej \cite{kolodziej}, Cegrell \cite{cegrell, cegrell1}, Demailly \cite{demailly1}, Fornaess-Sibony \cite{fornaess-sibony2}, and many others. In these works, the following monotone convergence is the corner stone: If $u_i^{(n)}$ is a sequence of  smooth psh functions decreasing to $u_i$, and if certain conditions are satisfied, then 
\begin{equation}
\lim _{n\rightarrow\infty}dd^cu_1^{(n)}\wedge \ldots \wedge dd^cu_j^{(n)}\wedge R
\label{EquationMonotoneConvergence}\end{equation}
exists, and the limit is independent of the choice of $(u_i^{(n)})$. We then define $dd^cu_1\wedge \ldots \wedge dd^cu_j\wedge R$ to be the limit. 

For the special case of projective spaces, a satisfying theory for intersection of general positive closed currents was given by Dinh-Sibony \cite{dinh-sibony4}, using superpotentials. This theory was extended to the case of compact K\"ahler manifolds in their next paper \cite{dinh-sibony3}. In their recent work \cite{dinh-sibony2}, they proposed a new approach using tangent currents.

In all of these works, the resulting intersection, whenever defined, is always positive.  

The problem of intersection of currents is, however, still very delicate, despite many efforts. For example, we may ask the following question:

{\bf Question 1.} If $D$ is a curve in a manifold $X$ of dimension $2$, can we define reasonably the wedge product $[D]\wedge [D]$?

If we want the resulting wedge product to be a {\bf positive measure}, then as shown by Shiffman and Taylor the answer to Question 1 and similar questions, is negative. 

In the local setting, Bedford-Taylor \cite{bedford-taylor1} defined the so-called non-pluripolar Monge-Ampere $(dd^c)^n$ as a positive measure, having no mass on pluri-polar sets. Their construction was developed in the global case of compact K\"ahler manifolds in Guedj-Zeriahi \cite{guedj-zeriahi} and   Boucksom-Eyssidieux-Guedj-Zeriahi \cite{boucksom-eyssidieux-guedj-zeriahi} to solve non-pluripolar Monge-Ampere equations whose right hand side is an arbitrary positive measure with no mass on pluripolar sets. Under these definitions, the intersection $[D]\wedge [D]$ in Question 1 is zero. In the global case, however, the non-pluripolar intersection product is not compatible with the intersection in cohomology. 

In this paper, we seek to define an intersection product which is compatible with the intersection in cohomology and is as close to positive measures as possible. 

\subsection{Main idea}

We discuss here the main idea, in the special case where $R=[D]$ is the current of integration over a curve $D$, and $T$ is a positive closed $(1,1)$ current. We write $T=\theta +dd^cu$, for a smooth closed $(1,1)$ form $\theta$ and a quasi-psh function $u$.

If $u$ is not identifically $-\infty$ on $D$, then the restriction $u|_D$ is a quasi-psh function on $D$, and we can define
\begin{eqnarray*}
T\wedge [D]:=\theta \wedge [D]+dd^c(u|_D).
\end{eqnarray*}
However, if $u=-\infty$ on $D$ then the term $dd^c(u|_D)$ is not defined. We may attempt to define, for example,  $dd^c(u|_D)=0$, arguing that for the approximation $u_j=\max (u,-j)$ of $u$, then $dd^c(u_j|_D)=0$ for all $j$. This would give $T\wedge [D]=\theta \wedge [D]$. However, if we use another representation $T=\theta '+dd^cu'$, then we again have $u'=-\infty$ on $D$, and thus must define $T\wedge [D]=\theta '\wedge [D]$. But this gives two different values for $T\wedge [D]$. 

The above discussion shows the difficulty when trying to define $T\wedge [D]$ as a measure with mass $\{T\}.\{D\}$, in the case where the quasi-potentials of $T$ are $-\infty$ on $D$. We thus need to proceed differently. Our main idea is first to identify the set of measures which can be reasonably associated with $T\wedge [D]$, and then to choose an appropriate supremum of these measures. In a certain sense, the supremum is less negative than any of the measures $\mu$.

For the first step, we take the monotone convergence in Equation (\ref{EquationMonotoneConvergence}) as the starting point. We then consider all signed measures $\mu$ of the form
\begin{eqnarray*}
\mu =\lim _{n\rightarrow\infty}T_n\wedge [D],
\end{eqnarray*}
where $T_n$ is an appropriate smooth approximation of $T$. 

For the second step, we need to be able to take a suppremum of the measures $\mu$. To this end, we need to have an order on such measures. We note that, in general, the set of measures $\mu$ considered above is not bounded (see Example \ref{Example1}). We also note that if $\mu$ and $\mu '$ are such two measures, then they have the same total mass, which is the intersection in cohomology $\{T\}.\{D\}$.  Therefore $\mu -\mu '$ is neither positive or negative, except if they are the same. Hence the usual order on measures is not useful here. In stead, the following partial order on signed measures is more suitable to our purpose.
\begin{definition}
Let $X$ be a compact K\"ahler manifold. 

1) Let $\mu =\mu ^+-\mu ^-$ be a signed measure on $X$, where $\mu ^{\pm}$ are positive measures with total masses $||\mu ^{\pm}||$. We denote $||\mu ||_{neg}$ to be the minimum of  $||\mu ^{-}||$, where $\mu ^{-}$ runs over all decomposition of $\mu$ as $\mu ^+-\mu ^-$.

2) Let $\mu _1$ and $\mu _2$ be signed measures on $X$. We write $\mu _1\succ \mu _2$ if either 

a) $||\mu _1||_{neg}<||\mu _2||_{neg}$.

or

b) $||\mu _1||_{neg}=||\mu _2||_{neg}$ and $\mu _1-\mu _2$ is a positive measure.

\label{DefinitionLeastNegativeMass}\end{definition} 
\begin{remark}
1) It is easy to check that $\succ$ is transitive. 

2) If $\mu _1-\mu _2$ is a positive measure then $||\mu _1||_{neg}\leq ||\mu _2||_{neg}$.  Note also that $||\mu ||_{neg}=0$ if and only if $\mu$ is itself a positive measure. 

Hence if $\mu _1\succ \mu _2$, then $\mu _1$ can be regarded to be closer to positive measures than $\mu _2$.  
\end{remark}

\subsection{Main results}

In the next section, we will define an intersection product $\bigwedge (T_1,T_2,\ldots ,T_{k-p},R)$ for a positive closed $(p,p)$ current $R$ and positive closed $(1,1)$ currents $T_1,\ldots ,T_{k-p}$, which is least negative in a certain sense to be specified later (see part 8) in Theorem \ref{TheoremMain1}). We now state the main results of the paper. 

\begin{theorem}
Let $X$ be a compact K\"ahler manifold of dimension $k$. Let $R$ be a positive closed $(p,p)$ current, and let $T_1,\ldots ,T_{k-p}$ be positive closed $(1,1)$ currents. 

Let $C^0(X)$ be the set of continuous functions $\varphi :X\rightarrow \mathbb{R}$. There is an operator $\bigwedge (T_1,\ldots ,T_{k-p},R):~C^0(X)\rightarrow \mathbb{R}$, which we call the least negative intersection, with the following properties. (In the below, we use the following notation, even when $\bigwedge (T_1,\ldots ,T_{k-p},R)$ is not a measure:
\begin{eqnarray*}
\int _X\varphi \bigwedge (T_1,\ldots ,T_{k-p},R):=<\bigwedge (T_1,\ldots ,T_{k-p},R),\varphi >.)
\end{eqnarray*} 

1) Sublinearity and boundedness: $\bigwedge (T_1,\ldots ,T_{k-p},R)$ is sublinear and bounded. More precisely, there is a constant $C>0$, independent of $T_1,\ldots ,T_{k-p}$ and $R$ such that for any $\varphi \in C^0(X)$ we have
\begin{eqnarray*}
|\int _X\varphi\bigwedge (T_1,\ldots ,T_{k-p},R)|~\leq C||\varphi ||_{L^{\infty}}||T_1||\times \ldots \times ||T_{k-p}||\times ||R||. 
\end{eqnarray*}
Here $||T_j||$ and $||R||$ are the total masses against a fixed K\"ahler form on $X$ of the positive currents $T_j$ and $R$.

Moreover, if $\lambda\in \mathbb{R}_{\geq 0}$, $B$ is a real number, and $\varphi _1,\varphi _2\in C^0(X)$ then 
\begin{eqnarray*}
\int _X\lambda\varphi _1\bigwedge (T_1,\ldots ,T_{k-p},R)&=&\lambda \int _X\varphi _1\bigwedge (T_1,\ldots ,T_{k-p},R),\\
\int _X(\varphi _1+\varphi _2)\bigwedge (T_1,\ldots ,T_{k-p},R)&\leq&\int _X\varphi _1\bigwedge (T_1,\ldots ,T_{k-p},R)+\int _X\varphi _2\bigwedge (T_1,\ldots ,T_{k-p},R),\\
\int _X(\varphi +B)\bigwedge (T_1,\ldots ,T_{k-p},R)&=&\int _X\varphi \bigwedge (T_1,\ldots ,T_{k-p},R)+B\int _X\bigwedge (T_1,\ldots ,T_{k-p},R).
\end{eqnarray*}

2) It is symmetric in $T_1,\ldots ,T_{k-p}$.

3) It is independent of the choice of a K\"ahler form on $X$.

4) Its total mass $<\bigwedge (T_1,\ldots ,T_{k-p},R),1>$ is the intersection in cohomology $\{T_1\}\ldots \{T_{k-p}\}.\{R\}$. 

5) If $U\subset X$ is an open set where the monotone convergence Equation (\ref{EquationMonotoneConvergence}) is satisfied for $T_1|_U,\ldots ,T_{k-p}|_U$ and $R|_U$, then $$\bigwedge (T_1,\ldots ,T_{k-p},R)|_U=T_1|_U\wedge \ldots \wedge T_{k-p}|_U\wedge R|_U,$$
where the right hand side is the current defined by the monotone convergence. 

6) Let $E_{>0}(T_j)$ be the set of points $x\in X$ for which the Lelong number $\nu (T_j,x)$ is positive. Assume that for any $1\leq i_1<\ldots <i_q\leq k-p$, the number 
\begin{eqnarray*}
\mathcal{H}_{2k-2p-2q+1}(E_{>0}(T_{i_1})\cap \ldots \cap E_{>0}(T_{i_q})\cap \sup (R))=0.
\end{eqnarray*}
Here $\mathcal{H}$ is the Hausdorff dimension.

Then $\bigwedge (T_1,\ldots ,T_{k-p},R)$ is positive, that is $<\bigwedge (T_1,\ldots ,T_{k-p},R),\varphi >$ is non-negative whenever $\varphi$ is non-negative.

7) (Stronger form of part 6) .) Assume that $R=[W]$ is the current of integration on a variety (need not to be irreducible) $W\subset X$. Let $E_{>0}(T_j)$ be the set of points $x\in X$ for which the Lelong number $\nu (T_j,x)$ is positive. Assume that for any $1\leq i_1<\ldots <i_q\leq k-p$, the set 
\begin{eqnarray*}
E_{>0}(T_{i_1})\cap \ldots \cap E_{>0}(T_{i_q})\cap W,
\end{eqnarray*}
does not contain any variety of dimension $>k-p-q$. Then $\bigwedge (T_1,\ldots ,T_{k-p},R)$ is positive.

8) For $j=1,\ldots ,k-p$, let us write $T_j=\theta _j+dd^cu_j$ for a smooth closed $(1,1)$ form $\theta _j$ and a quasi-psh function $u_j$. Let $u_j^{(n)}$ be sequences of smooth functions decreasing to $u_j$. Assume that that $\theta _j+dd^cu_j^{(n)}\geq \Omega $ for all $j$ and $n$, here $\Omega$ is a smooth closed $(1,1)$ form on $X$. Let $\mu$ (necessarily a signed measure) be a cluster point of the sequence $(\theta _1+dd^cu_1^{(n)})\wedge \ldots\wedge (\theta _{k-p}+dd^cu_{k-p}^{(n)})\wedge R$. Then there is a signed measure $\mu '$ of the same total mass as that of $\mu$, such that $\mu '\succ \mu$ and $\bigwedge (T_1,\ldots ,T_{k-p},R)-\mu '$ is positive. Here $||.||_{neg}$ is given in Definition \ref{DefinitionLeastNegativeMass}.
\label{TheoremMain1}\end{theorem} 

We note that $\bigwedge (T_1,\ldots ,T_{k-p},R)$ is a measure if and only if the set $\mathcal{L}_{T_1,\ldots ,T_{k-p},R}$ in Definition \ref{DefinitionMeasures} has only one element (if two measures have the same  total mass then not one of them can dominate the other, except if they are the same). Here we give one criterion to check whether $\bigwedge (T_1,\ldots ,T_{k-p},R)$ is a positive measure. 

\begin{proposition}
Assume that the number $\kappa _{T_1,\ldots ,T_{k-p},R}$ in Definition \ref{DefinitionMeasures} is $0$ (this is guaranteed for example when conditions similar to those in Corollary \ref{Corollary1} are satisfied). Assume also that there is a Zariski open set $U\subset X$ over which the monotone convergence in Equation (\ref{EquationMonotoneConvergence}) is satisfied for $T_1|_U,\ldots ,T_{k-p}|_U,R|_U$, such that the mass of $T_1|_U\wedge \ldots \wedge T_{k-p}|_U\wedge R|_U$ is exactly $\{T_1\}\ldots \{T_{k-p}\}.\{R\}$. Then 
\begin{eqnarray*}
\bigwedge (T_1,\ldots ,T_{k-p},R)=(T_1|_U\wedge \ldots \wedge T_{k-p}|_U\wedge R|_U)^o,
\end{eqnarray*}
here the right hand side is the extension by $0$.
\label{Proposition0}\end{proposition} 

We give here one example illustrating Proposition \ref{Proposition0}. 
\begin{example}
We let $J:\mathbb{P}^2\rightarrow \mathbb{P}^2$ be the map $$J[x_0:x_1:x_2]=[\frac{1}{x_0}:\frac{1}{x_1}:\frac{1}{x_2}].$$
Let $e_0=[1:0:0]$, $e_1:=[0,1,0]$, $e_2=[0:0:1]$. Let $\omega _{\mathbb{P}^2}$ be the Fubini-Study form on $\mathbb{P}^2$. For any point $p\in \mathbb{P}^2$ in generic position, there is a positive closed $(1,1)$ current $T(p)$ with the following properties:

i) The cohomology class of $T(p)$ is $2\omega _{\mathbb{P}^2}$,

ii) $T(p)$ is smooth outside $e_0,e_1,e_2,p$,

and

iii) $T(p)$ has Lelong number $1$ at $e_0,e_1,e_2$ and $p$.
 
Such a current can be constructed by averaging over curves of degree $2$ in $\mathbb{P}^2$ containing $e_0,e_1,e_2$ and $p$. 
 
We choose a finite number of such points $p_1,\ldots ,p_m$. Choose $\nu _1,\ldots ,\nu _m$ be positive numbers such that $\nu _1+\ldots +\nu _m=1$. 

We let $T=\nu _1T(p_1)+\ldots +\nu _mT(p_m)$.

Let $U=\mathbb{P}^2-\{x_0x_1x_2\not= 0\}$. Then $J:U\rightarrow U$ is an isomorphism. We let $$S=J^o(T)=(J|_U^*(T|_U))^o$$ be the strict transform of the current $T$. Then $S$ has the same cohomology class as that of $\omega _{\mathbb{P}^2}$, and $\bigwedge(S,S)=J^*(T\wedge T|_U)$ is a positive measure. Both its pluripolar part (i.e. Dirac masses) and non-pluripolar part are non-zero. A priori, the current $S$ and the measure $J^*(T\wedge T|_U)$ are quite singular near the curves $x_0x_1x_2=0$. 
\label{Example2}\end{example}

By Theorem \ref{TheoremMain1}, if $T_1,\ldots ,T_k$ are  positive closed $(1,1)$ currents, there is a symmetric least negative wedge product $\bigwedge (T_1,\ldots ,T_k)$. We obtain the following
\begin{corollary}
1) Let $E_{>0}(T_j)$ be the set where the Lelong number of $T_j$ is positive. If for every $1\leq i_1<\ldots <i_q\leq k$, the set $E_{>0}(T_{i_1})\cap \ldots \cap E_{>0}(T_{i_q})$ does not contain any subvariety $W\subset X$ of dimension $>k-q$, then $\bigwedge (T_1,\ldots ,T_k)$ is positive. 

2) If $X=\mathbb{P}^k$ then $\bigwedge (T_1,\ldots ,T_k)$ is positive.
\label{Corollary1}\end{corollary}
\begin{proof}

1) Follows from part 7) of Theorem \ref{TheoremMain1}.

2) Follows from the fact that any positive closed $(1,1)$ current $T=\theta +dd^cu$ on $X=\mathbb{P}^k$ is the limit of a sequence of positive closed smooth forms $T_n=\theta +dd^cu_n+\epsilon _n\omega _X$, where $u_n$ decreases to $u$ and $\epsilon _n$ decreases to $0$.  
\end{proof}

In particular, in the case $T_1=\ldots =T_k=T$, the least negative Monge-Ampere $MA(T)=\bigwedge (T,\ldots ,T)$ is positive if $E_{>0}(T)$ does not contain any curve. There are many examples where $T$ are currents of integration over a hypersurface and $MA(T)$ has negative global mass. The following gives one example of positive closed $(1,1)$ currents $T$ in dimension $3$ such that $T$ is smooth outside a curve $D$, but $MA(T)$ has total negative mass and has support on the curve. 
\begin{theorem}
Let $J_X:X\rightarrow X$ be the bimeromorphic map in Theorem 1.6 in \cite{truong}. There is a positive closed smooth $(1,1)$ form $\alpha$ on $X$ such that 

i) The cohomology class of $\alpha$ is nef, but that of $J_X^*(\alpha )$ is not nef. In fact, in cohomology $\{J_X^*(\alpha )\}.\{J_X^*(\alpha )\}.\{J_X^*(\alpha )\}=-3$.

ii) The Monge-Ampere operator $MA(J_X^*(\alpha ))$ has support in the indeterminate set $I(J_X)$ and has total mass $-3$.
\label{Theorem4}\end{theorem}

\begin{remark}
There have been many works on solving Monge-Ampere equations $MA(T)=\mu$ where $\mu$ is a positive measure with support on a pluripolar set (Lempert \cite{lempert, lempert1}, Celik-Poletsky \cite{celik-poletsky}, Demailly \cite{demailly2}, Lelong \cite{lelong}, Zeriahi \cite{zeriahi}, Xing \cite{xing}, Ahag-Cegrell-Czyz-Pham \cite{ahag-cegrell-czyz-pham},...). However, Theorem \ref{Theorem4} is different in that the total mass is negative.  
\end{remark} 

{\bf Acknowledgements.} The author would like to thank Viet-Anh Nguyen for his helpful comments on an earlier version of the paper.

\section{Least negative intersection}   

In this section we define the least negative intersection $\bigwedge (T_1,\ldots ,T_{k-p},R)$ and prove the results stated in the introduction.

Let us first recall some notations. The interested readers can consult for example the book Demailly \cite{demailly}. Let $X$ be a compact K\"ahler manifold of dimension $k$. An upper-semicontinuous  function $u$ on $X$ is quasi-psh if it is integrable and there is a smooth closed $(1,1)$ form $\alpha$ such that $T=\alpha +dd^cu$ is a positive current. In this case, we say that $u$ is a quasi-potential for $T$.    

For a quasi-psh function $u$, its Lelong number is defined as follows: For $x\in X$
$$\nu (u,x) =\liminf _{z\rightarrow x} \frac{u(z)}{\log |z-x|}.$$
If $T$ is a positive closed $(1,1)$ current and $u$ is a quasi-potential of $T$, then we define $\nu (T,x)=\nu (u,x)$.

We will use the following result, due to Demailly \cite{demailly1}. We recall that a quasi-psh function $u$ has analytic singularities if locally it can be written as 
\begin{eqnarray*}
u=\gamma +\frac{c}{2}\log (|f_1|^2+\ldots +|f_m|^2),
\end{eqnarray*}
where $\gamma$ is a smooth function, $c>0$ is a constant, $f_1,\ldots ,f_m$ are holomorphic functions.

\begin{theorem}
Let $X$ be a compact K\"ahler manifold with a K\"ahler form $\omega _X$. Let $T=\theta +dd^cu$ be a positive closed $(1,1)$ current on $X$.

1) There are a sequence of smooth functions $(u_n)$ decreasing to $u$, and a positive number $A>0$ such that $\theta +dd^cu_n\geq -A\omega _X$ for all $n$.

2) There are a sequence of quasi-psh functions $u_n$  with analytic singularities decreasing to $u$, and a sequence $\epsilon _n$ of positive numbers decreasing to $0$, such that the following are satisfied:

i) $\theta +dd^cu_n\geq -\epsilon _n\omega _X$ for all $n$,

and

ii) $\nu (u_n,x)$ increases to $\nu (u,x)$ uniformly with respect to $x\in X$.   
\label{PropositionDemailly}\end{theorem}

We now proceed to define least intersection of positive closed currents. We first define a set of good quasi-psh functions with respect to a positive closed $(p,p)$ current $R$.

\begin{definition}
Let $X$ be a compact K\"ahler manifold. Let $R$ be a positive closed $(p,p)$ current on $X$. We define $\mathcal{E}(R)$ to be the set of all $(k-p)$-tuples $(u_1,\ldots ,u_{k-p})$ where $u_j$ are quasi-psh functions with the following property:

For any $1\leq i_1<\ldots <i_q\leq k-p$, and for any sequence of smooth functions $(u_{i_j}^{(n)})$ decreasing to $u_{i_j}$ such that $dd^cu_{i_j}^{(n)}\geq \Omega$ for all $j$ and $n$, here $\Omega$ is a smooth closed $(1,1)$ form on $X$, then the following limit
\begin{eqnarray*}
\lim _{n\rightarrow\infty}dd^cu_{i_1}^{(n)}\wedge \ldots \wedge dd^cu_{i_q}^{(n)}\wedge R,
\end{eqnarray*}
exists, and is independent of the choice of $(u_{i_j}^{(n)})$. We then define $dd^cu_{i_1}\wedge \ldots dd^cu_{i_q}\wedge R$ to be the limit. 
\label{DefinitionClassE}\end{definition}

We have the following simple observation, concerning monotone convergence in $\mathcal{E}(R)$.
\begin{lemma}

Let $(u_1^{(n)},\ldots ,u_{k-p}^{(n)})$ be a sequence in $\mathcal{E}(R)$. Assume that $u_j^{(n)}$ decreases to a quasi-psh function $u$. Moreover, assume that there is a smooth closed $(1,1)$ form $\Omega$ on $X$ such that $dd^cu_j^{(n)}\geq \Omega $ for all $j$ and $n$.

1) If $(u_1,\ldots ,u_{k-p})\in \mathcal{E}(R)$, then for every $1\leq i_1<\ldots <i_q\leq k-p$ we have
\begin{eqnarray*}
\lim _{n\rightarrow\infty}dd^cu_{j_1}^{(n)}\wedge \ldots \wedge dd^cu_{i_q}^{(n)}\wedge R=dd^cu_{i_1}\wedge \ldots \wedge dd^cu_{i_q}\wedge R.
\end{eqnarray*}  

2) If 
\begin{eqnarray*}
\mu =\lim _{n\rightarrow\infty}dd^cu_{1}^{(n)}\wedge \ldots \wedge dd^cu_{k-p}^{(n)}\wedge R,
\end{eqnarray*}
then there is a sequence of smooth functions $(v_{1}^{(n)},\ldots ,v_{k-p}^{(n)})$ with the following properties:

a) $v_{j}^{(n)}$ decreases to $u_{j}$ for all $j$,

b) There is a smooth closed $(1,1)$ form $\Omega '$ such that $dd^cv_{j}^{(n)}\geq \Omega '$ for all $j$ and $n$,

and 

c)
\begin{eqnarray*}
\mu =\lim _{n\rightarrow\infty}dd^cv_{1}^{(n)}\wedge \ldots \wedge dd^cv_{k-p}^{(n)}\wedge R.
\end{eqnarray*}
\label{LemmaClassEIsClosed}\end{lemma}
\begin{proof}

By Definition of $\mathcal{E}(R)$, we see that 1) follows from 2). Hence it suffices to prove 2).

By Demailly's regularization theorem \cite{demailly1} (which was recalled in Theorem \ref{PropositionDemailly}), for any $j$ and $n$, there are sequences $\Phi _m(u_{j}^{(n)})$ of smooth functions with the following properties: 

i) $\Phi _m(u_{j}^{(n)})$ decreases to $u_{j}^{(n)}$,

and 

ii) $dd^c\Phi _m(u_{j}^{(n)})\geq \Omega -A\omega _X$ for all $j,n,m$. 

Since $X$ is compact, the space $C^0(X)$, equipped with the $L^{\infty}$ norm, is separable. Therefore there is a dense countable set $\mathcal{F}\subset C^0(X)$. We enumerate the elements in $\mathcal{F}$ as $\varphi _1,\varphi _2,\ldots $

For any number $l$, there is a number $n_l$ such that for all $n\geq n_l$ and for all $\varphi \in \{\varphi _1,\ldots ,\varphi _l\}$ 
\begin{equation}
|\int _X\varphi \mu -\int _X\varphi dd^cu_{1}^{(n)}\wedge \ldots \wedge dd^cu_{k-p}^{(n)}\wedge R|\leq \frac{1}{l}.
\label{EquationMinus}\end{equation}

We can choose $n_l$ such that $n_1<n_2<\ldots $. Then for each $j$, the sequence $u_j^{(n_l)}$ decreases to $u_j$. Therefore, we may assume that $||u_j^{(n_l)}-u_j||_{L^1(X)}\leq 1/l^2$ for all $j$ and $l$.

Since $(u_1^{(n_l)},\ldots ,u_{k-p}^{(n_l)})\in \mathcal{E}(R)$, for each $l$ there is a number $m_l$ such that for all $m\geq m_l$ and for all $\varphi \in \{\varphi _1,\ldots ,\varphi _l\}$ 
 \begin{equation}
|\int _X\varphi dd^c\Phi _m(u_1^{(n_l)})\wedge \ldots \wedge dd^c\Phi _m(u_{k-p}^{(n_l)})\wedge R -\int _X\varphi dd^cu_{1}^{(n_l)}\wedge \ldots \wedge dd^cu_{k-p}^{(n_l)}\wedge R|\leq \frac{1}{l}.
\label{Equation0}\end{equation}

Now we choose the sequence of smooth functions $w_j^{(l)}$ decreasing to $u_j$ as follows: 

Choose $w_j^{(1)}=\Phi _{m_1}(u_j^{(n_1)})$. We also arrange so that $||w_j^{(1)}-u_j^{(n_1)}||_{L^1(X)}\leq 1$.

Since $w_j^{(1)}\geq u_j^{(n_1)}\geq u_j^{n_2}$, $w_j^{(1)}$ is smooth,  $\Phi _m(u_1^{(n_2)})$ decreases to $u_j^{(n_2)}$, and $dd^c\Phi _m(u_{j}^{(n_2)})\geq \Omega -A\omega _X$ for all $m$ and $j$, by Hartogs' lemma we have
\begin{eqnarray*}
w_j^{(1)}+1\geq \Phi _m(u_1^{(n_2)}),
\end{eqnarray*}
provided $m$ is large enough. We the choose $m$ large enough such that $w_j^{(2)}=\Phi _m(u_1^{(n_2)})$ satisfies Equation (\ref{Equation0}) and
\begin{eqnarray*}
w_j^{(2)}&\leq& w_j^{(1)}+1,\\
||w_j^{(2)}-u_j^{(n_2)}||_{L^1(X)}&\leq&\frac{1}{2^2},
\end{eqnarray*}

Constructing inductively, we can find a sequence of smooth functions $w_j^{(l)}=\Phi _m(u_1^{(n_l)})$ for some large enough such that Equation (\ref{Equation0}) is satisfied and
\begin{eqnarray*}
w_j^{(l)}&\leq& w_j^{(l-1)}+1/l^2,\\
||w_j^{(l)}-u_j^{(n_l)}||_{L^1(X)}&\leq&\frac{1}{l^2}.
\end{eqnarray*}

Since $dd^cw_j^{(l)}\geq \Omega -A\omega _X$ for all $j$ and $l$, we can assume that there is a signed measure $\mu '$ such that 
\begin{eqnarray*}
\mu '=\lim _{l\rightarrow\infty}dd^cw_1^{(l)}\wedge \ldots \wedge dd^cw_{k-p}^{(l)}\wedge R.
\end{eqnarray*}

By Equations (\ref{EquationMinus}) and (\ref{Equation0}), for all $\varphi \in \mathcal{F}$
\begin{eqnarray*}
\int _X\varphi \mu =\int _X\varphi \mu '.
\end{eqnarray*}
Since $\mathcal{F}$ is dense in $C^0(X)$, it follows that $\mu =\mu '$.  

Finally, we define
\begin{eqnarray*}
v_j^{(l)}=w_j^{(l)}+\sum _{h=l}^{\infty}\frac{1}{h^2}.
\end{eqnarray*}
Then $v_j^{(l)}\geq v_j^{(l+1)}$ and $dd^cv_j^{(l)}=dd^cw_j^{(l)}\geq \Omega$. From the $L^1$ norm bounds $||u_j^{(n_l)}-u_j||_{L^1(X)}\leq 1/l^2$ and $||w_j^{(l)}-u_j^{(n_l)}||_{L^1(X)}\leq 1/l^2$, we obtain that $v_j^{(l)}$ decreases to $u_j$. 
\end{proof}

Based on this, we give the following definition.
\begin{definition}
Let $X$ be a compact K\"ahler manifold of dimension $k$. Let $R$ be a positive closed $(p,p)$ current and $T_1,\ldots ,T_{k-p}$ positive closed $(1,1)$ currents on $X$. For $j=1,\ldots ,k-p$, we write $T_j=\theta _j+dd^cu_j$ where $\theta _j$ is a smooth closed $(1,1)$ form and $u_j$ is a quasi-psh function. 

We let $\mathcal{G}_{T_1,\ldots ,T_{k-p},\theta _1,\theta _2,\ldots ,\theta _j,u_1,\ldots ,u_j, R}$ be the set of signed measures on $X$ of the form
\begin{eqnarray*}
\mu =\lim _{n\rightarrow\infty}T_1^{(n)}\wedge \ldots \wedge T_{k-p}^{(n)}\wedge R. 
\end{eqnarray*}
In the above, for $j=1,\ldots ,k-p$, $T_{j,n}=\theta _j+dd^cu_j^{(n)}$ where $u_j^{(n)}$ is a sequence of smooth functions decreasing to $u_j$ such that 

$$\theta _j+dd^cu_j^{(n)}\geq \Omega ,$$ for all $j,n$, here $\Omega$ is a smooth closed $(1,1)$ form on $X$. 
\label{DefinitionLeastNegativeMass1}\end{definition}

We observe that the set $\mathcal{G}_{T_1,\ldots ,T_{k-p},\theta _1,\theta _2,\ldots ,\theta _j,u_1,\ldots ,u_j, R}$ is independent of the choice of $\theta _1,\ldots ,\theta _j$ and $u_1,\ldots ,u_j$.
\begin{lemma}
In Definition \ref{DefinitionLeastNegativeMass1}, the set $\mathcal{G}_{T_1,\ldots ,T_{k-p},\theta _1,\theta _2,\ldots ,\theta _j,u_1,\ldots ,u_j, R}$  is independent of the choice of $\theta _j$ and $u_j$.
\label{LemmaGIndependent}\end{lemma}
\begin{proof}
Let us assume that we have two different ways to write $T_i$: $T_i=\theta _i+dd^cu_i=\theta _i'+dd^cu_i'$. Then we can write $u_i=u_i'+\phi _i$ where $\phi _i$ is a smooth function such that $dd^c\phi _i=\theta _i'-\theta _i$.

Let $(u_i^{(n)})$ be a sequence of smooth functions decreasing to $u$ such that $\theta _i+dd^cu_i^{(n)}\geq \Omega $ for every $i,n$, here $\Omega$ is a smooth closed $(1,1)$ form. 

Then $(u_i'^{(n)})=(u_i^{(n)}-\phi _i)$ is a sequence of smooth functions decreasing to $u_i-\phi _i=u_i'$. For each $n$ then $$dd^cu_i'^{(n)}=dd^cu_i^{(n)}-dd^c\phi _i\geq \Omega -dd^c\phi _i.$$ Since 
\begin{eqnarray*}
\theta _i'+dd^cu_i'^{(n)}=\theta _i'+dd^cu_i^{(n)}-dd^c\phi _i=\theta _i+dd^cu_i^{(n)}\geq \Omega ,
\end{eqnarray*}
for all $i$, it follows that $$\mathcal{G}_{T_1,\ldots ,T_{k-p},\theta _1,\theta _2,\ldots ,\theta _j,u_1,\ldots ,u_j, R}\subset \mathcal{G}_{T_1,\ldots ,T_{k-p},\theta  _1',\theta _2',\ldots ,\theta _j',u_1',\ldots ,u_j', R}.$$

Exchanging the roles of $(\theta _i,u_i)$ and $(\theta _i',u_i')$, we obtain the reverse inclusion
$$\mathcal{G}_{T_1,\ldots ,T_{k-p},\theta _1',\theta _2',\ldots ,\theta _j',u_1',\ldots ,u_j', R}\subset \mathcal{G}_{T_1,\ldots ,T_{k-p},\theta _1',\theta _2',\ldots ,\theta _j',u_1',\ldots ,u_j', R}.$$
\end{proof}

By Lemma \ref{LemmaGIndependent}, there is a well-defined set $\mathcal{G}_{T_1,\ldots ,T_{k-p},R}$ of signed measures on $X$, depending only on the currents $T_1,\ldots ,T_{k-p}$ and $R$. In general, the set $\mathcal{G}_{T_1,\ldots ,T_{k-p},R}$ is not bounded. We have the following example.
\begin{example}
Let $D$ be an irreducible curve in $X$, and let $T$ be a positive closed $(1,1)$ current. Assume that one (hence every) quasi-potential $u$ of $T$ has the following properties: $e^u$ is continuous near $D$, and $u|_{D}=-\infty$. Then $\mathcal{G}_{T_1,\ldots ,T_{k-p},R}$ contains all measures of the form $\theta \wedge [D]$, here $\theta$ is a smooth closed $(1,1)$ form with the same cohomology class as that of $T$.
\label{Example1}\end{example}
\begin{proof}
Let $\theta$ be a smooth closed $(1,1)$ form with the same cohomology class as that of $T$. We write $T=\theta +dd^cu$. Consider the sequence $u_n=\max \{u,-n\}$. If we choose $\omega _X$ be a positive closed smooth $(1,1)$ form such that $\omega _X\geq \theta$, then for every $n$ we have $dd^cu_n\geq -\omega _X$. The assumption that $e^u$ is continuous near $D$ implies that $dd^cu_n=0$ near $D$. Hence 
\begin{eqnarray*}
\lim _{n\rightarrow\infty}(\theta +dd^cu_n)\wedge [D]=\theta \wedge [D].
\end{eqnarray*}
Lemma \ref{LemmaClassEIsClosed} implies that $\theta \wedge [D]\in \mathcal{G}_{T_1,\ldots ,T_{k-p},R}$.
\end{proof}

\begin{definition}
 We define 
\begin{eqnarray*}
\kappa _{T_1,T_2,\ldots ,T_{k-p},R}=\inf _{\mu \in \mathcal{G}_{T_1,\ldots ,T_{k-p},R}}||\mu ||_{neg}. 
\end{eqnarray*}
Here $||.||_{neg}$ is given in Definition \ref{DefinitionLeastNegativeMass}. 

We let $\mathcal{L}_{T_1,\ldots ,T_{k-p},R}$ be the set of signed measures $\mu$ on $X$ such that
\begin{eqnarray*}
\mu =\lim _{j\rightarrow\infty}\mu _j,
\end{eqnarray*}
for some sequence $(\mu _j)\subset \mathcal{G}_{T_1,\ldots ,T_{k-p},R}$ such that
\begin{eqnarray*}
\lim _{j\rightarrow\infty}||\mu _j||_{neg}=\kappa _{T_1,\ldots ,T_{k-p},R}.
\end{eqnarray*}
\label{DefinitionMeasures}\end{definition} 

By definition, if $\mu\in \mathcal{L}_{T_1,\ldots ,T_{k-p},R}$, then $||\mu ||_{neg}\leq \kappa _{T_1,\ldots ,T_{k-p},R}$. However, a priori the strict inequality may happen.  The measures $\mu\in  \mathcal{L}_{T_1,\ldots ,T_{k-p},R}$ with $||\mu ||_{neg}<\kappa _{T_1,\ldots ,T_{k-p},R}$ are very special, as can be seen from the following lemma. 

\begin{lemma}
Let $\mu\in \mathcal{L}_{T_1,\ldots ,T_{k-p},R}$. Let
\begin{eqnarray*}
\mu =\lim _{j\rightarrow\infty}\mu _j,
\end{eqnarray*}
where for each $j$
\begin{eqnarray*}
\mu _j=\lim _{n\rightarrow\infty}(\theta _1+dd^cu_1^{(n,j)})\wedge \ldots \wedge (\theta _1+dd^cu_{k-p}^{(n,j)})\wedge R.
\end{eqnarray*}
In the above $||\mu _j||_{neg}\rightarrow \kappa _{T_1,\ldots ,T_{k-p},R}$, and $u_i^{(n,j)}$ are smooth functions decreasing to $u_i^{(j)}$ such that 
\begin{eqnarray*}
dd^cu_i^{(n,j)}\geq \Omega _j,
\end{eqnarray*}
for all $n,j$ where $\Omega _j$ is a smooth closed $(1,1)$ form depending on $j$.

Assume that we can choose $\Omega _j$ to be independent of $j$, that is there is a smooth closed $(1,1)$ form $\Omega$ such that $\Omega _1=\Omega _2=\ldots =\Omega $. Then $\mu \in \mathcal{G}_{T_1,\ldots ,T_{k-p},R}$. In particular, $||\mu ||_{neg}=\kappa _{T_1,\ldots ,T_{k-p},R}$.
\label{LemmaCharacterizingMeasures}\end{lemma}
\begin{proof}
That $\mu \in \mathcal{G}_{T_1,\ldots ,T_{k-p},R}$ can be proved by using the same argument as in the proof of Lemma \ref{LemmaClassEIsClosed}. 
\end{proof}

By definition, any measure $\mu \in \mathcal{L}_{T_1,\ldots ,T_{k-p},R}$ can be written as $\mu =\mu ^+-\mu ^-$, where $\mu ^{\pm}$ are positive measures of total masses 
\begin{eqnarray*}
\int _X\mu ^+=\{T_1\}\ldots \{T_{k-p}\}.\{R\}+c,~\int _X\mu ^-=c.
\end{eqnarray*}
Here $c$ satisfies $0\leq c\leq\kappa _{T_1,\ldots ,T_{k-p},R}$, the number $\kappa _{T_1,\ldots ,T_{k-p},R}$ being given in Definition \ref{DefinitionMeasures}.

By Theorem \ref{PropositionDemailly}, there is a positive number $A>0$ independent of $T_1,\ldots ,T_{k-p}$ and $R$ such that there is a measure $\mu \in \mathcal{G}_{T_1,\ldots ,T_{k-p},R}$ for which 
\begin{eqnarray*}
||\mu ||_{neg}\leq A||T_1||\times \ldots \times ||T_{k-p}||\times ||R||.
\end{eqnarray*}
This implies that $$\kappa _{T_1,\ldots ,T_{k-p},R}\leq A||T_1||\times \ldots \times ||T_{k-p}||\times ||R||.$$

Therefore, for every continuous function $\varphi :X\rightarrow \mathbb{R}$, the following number
\begin{equation}
<\bigwedge (T_1,\ldots ,T_{k-p},R),\varphi >:=\sup _{\mu \in \mathcal{L}_{T_1,\ldots ,T_{k-p},R}} \int _X\varphi \mu ,
\label{EquationDefinitionWedge}\end{equation}
is finite. 

Since all $\mu \in \mathcal{L}_{T_1,\ldots ,T_{k-p},R}$ has the same total mass $\mu \in \mathcal{L}_{T_1,\ldots ,T_{k-p},R}$, we see that for any constant $B$
\begin{eqnarray*}
<\bigwedge (T_1,\ldots ,T_{k-p},R),\varphi +B>=<\bigwedge (T_1,\ldots ,T_{k-p},R),\varphi >+B<\bigwedge (T_1,\ldots ,T_{k-p},R),1>. 
\end{eqnarray*}

The following result relates $\bigwedge (T_1,\ldots ,T_{k-p},R)$ to some other functions defined on $\mathcal{G}_{T_1,\ldots ,T_{k-p},R}$ only.
\begin{lemma}
Fix a smooth function $\varphi $ on $X$. For each $\epsilon >0$ we define
\begin{eqnarray*}
<\bigwedge (T_1,\ldots ,T_{k-p},R;\epsilon ),\varphi >:=\sup _{\mu \in \mathcal{G}_{T_1,\ldots ,T_{k-p},R},~||\mu ||_{neg}\leq \kappa _{T_1,\ldots ,T_{k-p},R} +\epsilon }\int _X\varphi \mu .
\end{eqnarray*}
Then $\lim _{\epsilon\rightarrow 0}\bigwedge (T_1,\ldots ,T_{k-p},R;\epsilon )=\bigwedge (T_1,\ldots ,T_{k-p},R)$.
\end{lemma}
\begin{proof}
By definition, if $\epsilon '>\epsilon >0$ then $\bigwedge (T_1,\ldots ,T_{k-p},R;\epsilon ')\geq \bigwedge (T_1,\ldots ,T_{k-p},R;\epsilon )$. Therefore, the limit 
\begin{eqnarray*}
\bigwedge (T_1,\ldots ,T_{k-p},R)'=\lim _{\epsilon\rightarrow 0}\bigwedge (T_1,\ldots ,T_{k-p},R;\epsilon )
\end{eqnarray*}
 exists. 
 
If $\mu \in \mathcal{L}_{T_1,\ldots ,T_{k-p},R}$, then there is a sequence $\mu _j\in \mathcal{G}_{T_1,\ldots ,T_{k-p},R}$ with $||\mu _j||_{neg}\leq \kappa _{T_1,\ldots ,T_{k-p},R} +1/j $ such that 
\begin{eqnarray*}
\mu =\lim _{j\rightarrow\infty}\mu _j\leq \lim _{j\rightarrow\infty}\bigwedge (T_1,\ldots ,T_{k-p},R;1/j )=\bigwedge (T_1,\ldots ,T_{k-p},R)'.
\end{eqnarray*}
Therefore
\begin{eqnarray*}
\bigwedge (T_1,\ldots ,T_{k-p},R)=\sup _{\mu \in \mathcal{L}_{T_1,\ldots ,T_{k-p},R}}\mu \leq \bigwedge (T_1,\ldots ,T_{k-p},R)'.
\end{eqnarray*}

Now we prove the reverse inequality. Let $\varphi$ be a positive function on $X$. For each $j$, we choose a measure $\mu _j\in \mathcal{G}_{T_1,\ldots ,T_{k-p},R}$ with $||\mu _j||\leq \kappa _{T_1,\ldots ,T_{k-p},R} +1/j $ such that 
\begin{eqnarray*}
|<\bigwedge (T_1,\ldots ,T_{k-p},R;1/j ),\varphi >-\int _X\varphi \mu |\leq 1/j.
\end{eqnarray*}
It follows that 
\begin{eqnarray*}
 <\bigwedge (T_1,\ldots ,T_{k-p},R)',\varphi >=\lim _{j\rightarrow\infty}\int _X\varphi \mu _j.
\end{eqnarray*}
Without loss of generality, we may assume that 
\begin{eqnarray*}
\lim _{j\rightarrow\infty}\mu _j=\mu . 
\end{eqnarray*}
Then, $\mu \in \mathcal{L}_{T_1,\ldots ,T_{k-p},R}$, and we obtain
\begin{eqnarray*}
<\bigwedge (T_1,\ldots ,T_{k-p},R)',\varphi >=\lim _{j\rightarrow\infty}\int _X\varphi \mu _j=\int _X\varphi \mu\leq \int _X\varphi \bigwedge (T_1,\ldots ,T_{k-p},R).
\end{eqnarray*}
 Therefore, $\bigwedge (T_1,\ldots ,T_{k-p},R)'\leq \bigwedge (T_1,\ldots ,T_{k-p},R)$, as wanted.
 \end{proof}

\section{Proofs of main results}
\begin{proof}[Proof of Theorem \ref{TheoremMain1}]

Let $\bigwedge (T_1,\ldots ,T_{k-p},R)$ be defined as in Equation (\ref{EquationDefinitionWedge}).

1) Follows easily from the definition. 

2) This follows since the definition of $\mathcal{G}_{T_1,\ldots ,T_{k-p},R}$,  and thus $\kappa _{T_1,\ldots ,T_{k-p},R}$ and $\mathcal{L}_{T_1,\ldots ,T_{k-p},R}$, is symmetric in $T_1,\ldots ,T_{k-p}$.

3) By definition, $\mathcal{G}_{T_1,\ldots ,T_{k-p},R}$ is independent of the choice of a K\"ahler form on $X$. Since the mass of a measure is independent of a K\"ahler form on $X$, the number $\kappa _{T_1,\ldots ,T_{k-p},R}$ is also independent of the choice of a K\"ahler form on $X$. Consequently, $\mathcal{L}_{T_1,\ldots ,T_{k-p},R}$ is also independent of the choice of a K\"ahler form on $X$.

4) For any $\mu \in \mathcal{G}_{T_1,\ldots ,T_{k-p},R}$, the total mass of $\mu$ is $\{T_1\}\ldots \{T_{k-p}\}.\{R\}$. Therefore, the total mass of $\bigwedge (T_1,\ldots ,T_{k-p},R)$ is
\begin{eqnarray*}
\int _{X}\bigwedge (T_1,\ldots ,T_{k-p},R)&=&\sup _{\mu \in \mathcal{L}_{T_1,\ldots ,T_{k-p},R}}\int _X\mu =\sup _{\mu \in \mathcal{L}_{T_1,\ldots ,T_{k-p},R}}\{T_1\}\ldots \{T_{k-p}\}.\{R\}\\
&=&\{T_1\}\ldots \{T_{k-p}\}.\{R\}.
\end{eqnarray*}

5) If $U\subset X$ is an open set over which the monotone convergence in Equation (\ref{EquationMonotoneConvergence}) holds, then for any $\mu \in \mathcal{G}_{T_1,\ldots ,T_{k-p},R}$ we have
\begin{eqnarray*}
\mu |_U=T_1|_U\wedge \ldots \wedge T_{k-p}|_U\wedge R|_{U}.
\end{eqnarray*}
From this, the claim follows.  

6) By Part 2) of Theorem \ref{PropositionDemailly}, for each $j$ there is a sequence $u_j^{(n)}$ of quasi-psh functions with analytic singularities and a sequence of positive numbers $\epsilon _n$ decreasing to $0$ (here $\epsilon _n$ can be chosen to be the same for all $j$)  such that:

i) $u_j^{(n)}$ decreases to $u_j$,

ii) $\theta _j+dd^cu_{j}^{(n)}\geq -\epsilon _n\omega _X$ for all $j,n$,

and

iii) The Lelong numbers of $u_j^{(n)}$ increases to the Lelong numbers of $u_j$.

Since $u_{j,n}$ has analytic singularities, for each $j$ and $n$ there is a subvariety $V_{j,n}\subset X$ which is exactly the set where the Lelong numbers of $u_j^{(n)}$ are positive
\begin{eqnarray*}
V_{j,n}=E_{>0}(u_j^{(n)}).
\end{eqnarray*}
Moreover, $u_{j,n}$ is smooth outside $V_{j,n}$.

(Here we used the following result on Lelong numbers of psh functions with analytic singularities. Let $$u(z)=\frac{1}{2}\log \sum _{j=1}^m|f_j|^2$$
be defined in an open set $0\in \Omega \subset \mathbb{C}^n$. Here $f_1,\ldots ,f_m$ are holomorphic functions, $f_1(0)=\ldots =f_m(0)=0$. Then the Lelong number of $u$ at $0$ is 
\begin{eqnarray*}
\nu (u,0)=\min _{j=1,\ldots ,m} \mbox{mult}_0(f_j)>0. 
\end{eqnarray*} 
)

Property iii) implies that $V_{j,n}\subset E_{>0}(u_j)$. The assumptions on $E_{>0}(u_j)$ implies that $(u_1^{(n)},\ldots ,u_{j}^{(n)})\in \mathcal{E}(R)$ (see for example Section 4, Chapter 3 in Demailly \cite{demailly}). 

Property ii) implies that any cluster point of the sequence $(\theta _1+dd^cu_1^{(n)})\wedge \ldots \wedge (\theta _1+dd^cu_{k-p}^{(n)})\wedge R$ is a positive measure. By  Lemma \ref{LemmaClassEIsClosed}, each of these clusters points is in $\mathcal{G}_{T_1,\ldots ,T_{k-p},R}$. Therefore $\bigwedge (T_1,\ldots ,T_{k-p},R)$ is positive. 

7) Let notations be as in the proof of part 6). Since $V_{j_1,n}\cap \ldots \cap V_{j_q,n}\cap W$ is a variety, the assumption on $E_{>0}(u_j)$ implies that the dimension of $V_{j_1,n}\cap \ldots \cap V_{j_q,n}\cap W$ is at most $k-p-q$. Therefore, we can apply again the results in Demailly \cite{demailly} to conclude that $(u_1^{(n)},\ldots ,u_{j}^{(n)})\in \mathcal{E}(R)$. Then, we proceed as in the proof of part 6).

8) Let $\mu$ be a measure in $\mathcal{G}_{T_1,\ldots ,T_{k-p},R}$. Then either $\mu \in \mathcal{L}_{T_1,\ldots ,T_{k-p},R}$ or $>\kappa _{T_1,\ldots ,T_{k-p},R}$. In the first case, it follows from definition that $\bigwedge (T_1,\ldots ,T_{k-p},R)-\mu \geq 0$. In the second case, we choose $\mu '$ to be any measure in $\mathcal{L}_{T_1,\ldots ,T_{k-p},R}$. Then $||\mu '||_{neg}<||\mu ||_{neg}$ and $\bigwedge (T_1,\ldots ,T_{k-p},R)-\mu '\geq 0$.
\end{proof}

\begin{proof}[Proof of Proposition \ref{Proposition0}]
By assumption on $\kappa _{T_1,\ldots ,T_{k-p},R}$, for every $\mu\in \mathcal{L} _{T_1,\ldots ,T_{k-p},R}$ then $\mu$ is a positive measure. Moreover, 
\begin{eqnarray*}
\mu |_U=T_1|_U\wedge \ldots \wedge T_{k-p}|_U\wedge R|_U.
\end{eqnarray*}
Therefore 
\begin{eqnarray*}
\mu \geq (T_1|_U\wedge \ldots \wedge T_{k-p}|_U\wedge R|_U)^o.
\end{eqnarray*}
Since the mass of the two measures in the above are the same, we conclude that they are the same. From this, the proposition follows.
\end{proof}

\end{document}